\documentclass[12pt]{amsart}

\usepackage{amsthm}
\usepackage{amsmath}
\usepackage{amssymb}
\usepackage{amssymb, amsfonts, amsbsy, latexsym, epsfig, color}
\usepackage{enumerate}

\newtheorem{thm}{Theorem}[section]
\newtheorem{theorem}[thm]{Theorem}

\newtheorem{corollary}[thm]{Corollary}

\newtheorem{lemma}[thm]{Lemma}
\newtheorem{proposition}[thm]{Proposition}

\newtheorem{definition}[thm]{Definition}
\theoremstyle{remark}
\newtheorem{remark}[thm]{Remark}

\newcommand{\ZZ}{\mathbb Z}

\newcommand{\spann}{\mbox{\rm span}}

\newcommand{\abs}[1]{\left| #1 \right|}

\title{An Elementary, Illustrative Proof of the Rado-Horn Theorem}
\author[P.G. Casazza, and J. Peterson
 ]{Peter G. Casazza and Jesse Peterson}
\address{Department of Mathematics, University
of Missouri, Columbia, MO 65211-4100}

\thanks{The first author is supported by NSF DMS 1008183, NSF ATD1042701
and AFOSR F1ATA00183G003}

\email{casazzap@missouri.edu;}

\begin{document}

\begin{abstract}
The Rado-Horn theorem provides necessary and sufficient conditions for 
when a collection of vectors can be partitioned into a fixed number of 
linearly independent sets.  Such partitions exist if and only if every
 subset of the vectors satisfies the so-called Rado-Horn inequality.  
Today there are at least six proofs of the Rado-Horn theorem, but these 
tend to be extremely delicate or require intimate knowledge of matroid theory.  In this paper we provide an elementary proof of the Rado-Horn theorem as well as elementary proofs for several generalizations including results for the redundant case when the hypotheses of the Rado-Horn theorem fail.  Another problem with the existing proofs of the
Rado-Horn Theorem is that they give no information about how to actually 
partition the vectors.
We start by considering a specific partition of the vectors, and the proof consists 
of showing
that this is an optimal partition. 
We further show how certain structures we construct in the proof are at 
the heart of the Rado-Horn theorem by characterizing subsets of vectors 
which maximize the Rado-Horn inequality. 
Lastly, we demonsrate how these results may be used to select an optimal partition 
with respect to spanning properties of the vectors.
 \end{abstract}

\maketitle


\section{Introduction}
The terminology {\it Rado-Horn theorem} was first introduced in \cite{bourgain}.  This theorem \cite{horn, rado} provides necessary and sufficient conditions for a collection of vectors to be partitioned into $k$ linearly independent sets:
\begin{theorem} (Rado-Horn)
Consider the vectors $\Phi=\{\varphi_i\}_{i=1}^M$ in a vector space.  
Then the follwing are equivalent.
\begin{enumerate}[(i)]
	\item The set $\Phi$ can be partitioned into sets $\{A_i\}_{i=1}^k$ 
such that $A_i$ is a linearly independent set for all $i=1,2,\ldots,k$.
	\item For any subset $J \subseteq \Phi$, we have 
$\abs{J}/\dim \spann (J) \leq k$.
\end{enumerate}
\end{theorem}

The Rado-Horn theorem has found application in several areas 
including progress on the Feichtinger conjecture \cite{CL}, a 
characterization of Sidon sets in $\Pi_{k=1}^{\infty}\ZZ_p$ 
\cite{MM, P}, and a notion of redundancy for finite frames \cite{BCK}.  
A generalized version of the Rado-Horn theorem has found use in frame 
theory as well where redundancy is at the heart of the subject \cite{BCPS}.  

Unfortunatly, proving the Rado-Horn theorem tends to very intricate or 
complicated.  Pisier, when discussing a characterization of Sidon sets 
in $\Pi_{k=1}^{\infty}\ZZ_p$ states
 ``$\ldots$ d'un lemme d'al\'{e}bre d\^{u} \`{a} Rado-Horn dont 
la d\'{e}monstration est relativement d\'{e}licate'' \cite{P}.  
Today there are at least six proofs of the Rado-Horn theorem 
\cite{CKS, CL, edmonds, HW, horn, rado}.  The theorem was proved 
in a more general algebraic setting in \cite{horn, rado} and then 
for matroids in \cite{edmonds}.  Each of these proofs are extremely 
delicate.  Harary and Welsh \cite{HW} improved upon the matroid version 
of the Rado-Horn theorem with a short and elegant proof; however, their
 argument requires a development of certain deep structures within
 matroid theory.  The Rado-Horn theorem was generalized in \cite{CKS}
 to include partitions of a collection of vectors with subsets of 
specified sizes removed, and the authors also proved results for the 
redundant case - the case where a collection of vectors cannot be 
partitioned into $k$ linear independent sets.  Unfortuneately the 
proofs for these refinements to the theorem are even more delicate 
than the original.  Finally, the Rado-Horn theorem was rediscovered 
in \cite{CL}, where the authors give an induction proof which may be 
considered elementary.  However, the proof has some limitations as it
 does not clearly generalize or describe the redundant case; it does 
not reveal the origin of the Rado-Horn inequality.  

In this paper, we present a elementary proof which is at the core of the Rado-Horn theorem.  With slight modification, these simple arguments prove a generalization of the Rado-Horn theorem and provide results for the redundant case similar to those in \cite{CKS}.  Perhaps most significanlty, the arguements we present may be thought of visually and provide insight into the specific conditions which give rise to the inequality in the Rado-Horn theorem.

This paper is organized into three sections.  The first develops constructions and main arguments used througout the paper.  The second section uses these tools to prove the Rado-Horn theorem and its generalization.  The final section describes which subsets maximise the Rado-Horn inequality and how this may be used to concretely construct a so-called fundamental partition.

\section{Preliminary Results}

Given a set of vectors, $\Phi=\{\varphi_i\}_{i=1}^M$, a main construction for our proofs will be a partition $F=\{F_i\}_{i=1}^{\ell}$ of $\Phi$ which is optimal in the sense that we have as many spanning sets as possible.  Then given this number of spanning sets, we have the partition also with the maximum number of sets spanning codimension one as possible.  This property continues through the entire partition of vectors.  We give a more formal definition below. 

\begin{definition}
Given vectors $\Phi=\{\varphi_i\}_{i=1}^M$, we call a partition $\{A_i\}_{i=1}^k$ of $\Phi$ an {\bf ordered partition} if $\abs{A_i}\geq \abs{A_{i+1}}$ for all $i=1,\ldots,k-1$.
\end{definition}

\begin{definition}
Given vectors $\Phi=\{\varphi_i\}_{i=1}^M$, let $\{P_i\}_{i=1}^m$ be all possible ordered partitions of $\Phi$ into linearly independent sets. Let $A_{ij}$ denote the $j$th set in the $i$th partition so that $P_i=\{A_{ij}\}_{i=1,j=1}^{m,r_i}$.
Now define
\[
	a_1=\sup_{i=1,\ldots,m}\abs{A_{i1}}.
\]
Then consider only the partitions $\{P_i : \abs{A_{i1}}=a_1\}$ and define
\[
		a_2=\sup_{\left\{i: \abs{A_{i1}}=a_1\right\}}\abs{A_{i2}}.
\]
We continue to consider fewer and fewer partitions so that given $a_1,\ldots,a_n$,
\[
	a_{n+1}=\sup_{\left\{i: \abs{A_{i1}}=a_1,\ldots,A_{in}=a_{n}\right\}}\abs{A_{i(n+1)}}.
\]
When $\sum_{i=1}^{\ell} a_i=M$, any remaining partition is in the set $\{P_i:\abs{A_{i1}}=a_1,\dots, \abs{A_{i\ell}}=a_{\ell}\}$. We call any such ordered partition of $\Phi$ into linearly indpendent sets $F=\{F_i\}_{i=1}^{\ell}$ a {\bf fundamental partition}.
\end{definition}

We define a fundamental partition in this manner simply because this definition makes existence clear.  However, the following theorem gives an equivalent definition and is Theorem 1 from \cite{colorings}.

\begin{theorem}
Let $\Phi=\{\varphi_i\}_{i=1}^M$ be a collection of vectors.  Then $F=\{F_i\}_{i=1}^{\ell}$ is a fundamental partition if and only if for any other ordered partition $\{P_i\}_{i=1}^k$ of $\Phi$ into linearly independent sets,
\begin{enumerate}[(i)]
	\item $\ell \leq k$
	\item	$\sum_{i=1}^j \abs{P_i} \leq \sum_{i=1}^j\abs{F_i}, j=1,2,\ldots,\ell$.
\end{enumerate}
\end{theorem}
That is, an ordered partition of $\Phi$ is a fundamental partition if and only if it {\it majorizes} every other ordered partition of $\Phi$ into linearly independent sets.  

It is helpful to view a fundamental partition as a Young diagram where each square represents a vector, and the rows correspond to the sets of the partition.  See 
figure \ref{fig1}.


\begin{figure}[ht!] 
   \center{\includegraphics*[scale = 0.7]{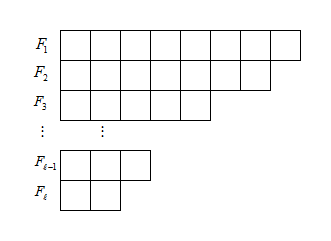}}
   \caption{Example of a fundamental partition}
  \label{fig1}
\end{figure}

We will occasionally need to move vectors within a fundamental partition.  If we do this carefully, the linear independence properties and the subspaces spanned by many of the sets in the partition will remain unchanged.

\begin{proposition}\label{switch}
Let $\Phi=\{\varphi_i\}_{i=1}^{M}$ be a collection of linearly independent vectors.  Suppose $\psi \in \spann{}(\Phi)$ so that $\psi = \sum_{i=1}^{M} c_i\varphi_i$.  Then for any $j\in \{1,\ldots, M\}$ such that $c_j \neq 0$, $\Psi_j=\{\Phi, \psi\} \setminus \{\varphi_j\}$ is linearly independent and $\spann{}(\Psi_j)=\spann{}(\Phi)$.
\end{proposition}

\begin{proof}
Suppose for some scalars $a_i$,
\begin{align*}
	\sum_{i=1,i\neq j}^M a_i\varphi_i + a_j\psi 	&= \sum_{i=1,i\neq j}^M a_i\varphi_i + a_j\sum_{i=1}^Mc_i\varphi_i\\
																								&= \sum_{i=1,i\neq j}^M (a_i+a_jc_i)\varphi_i + a_jc_j \varphi_j \\
																								&= 0.																	
\end{align*}
Since $c_j\neq 0$, we must have $a_j=0$ by linear independence of $\Phi$, but then $a_i=0$ for all $i\in\{1,\ldots j-1, j+1,\ldots,M\}$.  Thus $\Psi_j$ is linearly independent.

Since $\psi\in \spann{}(\Phi)$, then $\spann{}(\Psi_j)\subseteq \spann{}(\Phi)$. It follows $\spann{}(\Psi_j)=\spann{}(\Phi)$ since $\abs{\Psi_j}= \abs{\Phi}$.
\end{proof}

We will also be interested in which vectors are contained in the spans of each set in a fundamental partition.  The following lemma is trivial but does provides some information in this regard.

\begin{lemma} \label{up}
Let $\{F_i\}_{i=1}^{\ell}$ be a fundamental partition of $\Phi=\{\varphi_i\}_{i=1}^M$.  Then $\spann{}(F_j) \subseteq \spann{}(F_i)$ for $i\leq j$.
\end{lemma}

\begin{proof} 
Suppose there existed some $\varphi \in F_j$, such that $\varphi \notin \spann{}(F_i)$.  Then $\{F_{i}, \varphi\}$ would be linearly independent contradicting our assumption 
that $F$ is a fundamental partition.
\end{proof}

This shows that
 any vector is contained in the spans of the sets before it in the partition; however, we may carefully choose certain vectors which must be contained in the spans of all sets, with the possible exception of $F_{\ell}$.

\begin{lemma}\label{down}
Let $\{F_i\}_{i=1}^{\ell}$ be a fundamental partition of $\Phi=\{\varphi_i\}_{i=1}^M$.  Pick any $\varphi_{\ell} \in F_{\ell}$ and let $S_{\ell-1}\subseteq F_{\ell-1}$ be smallest set such that $\varphi \in \spann{}(S_{\ell-1})$.  Fix any $j\leq \ell-1$, and let $S_j$ be the smallest subset of $F_j$ such that $\spann{}(S_{\ell-1})\subseteq \spann{}(S_j)$.  Then $\spann{}(S_j)\subseteq \spann{}(F_{i})$, $i=1,\ldots,\ell-1$.
\end{lemma}

\begin{proof}
Clearly the sets $S_{\ell-1}$ and $S_j$ exist by Lemma \ref{up}.  We will prove the statement for $i=\ell-1$.  The result will then follow for all $i=1,\ldots,\ell-1$ since $\spann{}(F_{\ell-1})\subseteq \spann{}(F_i)$ for $i \leq \ell-1$.

We will assume the result fails and get a contradiction.
So assume there exists some $\varphi_{j} \in S_j$ such that $\varphi_{j} \notin \spann{}(F_{\ell-1})$.  By Proposition \ref{switch}, there exits some $\varphi_{\ell-1}\in S_{\ell-1}$ such that $\{S_j,\varphi_{\ell-1}\}\setminus \{\varphi_{j}\}$ is linearly independent with the same span as $S_j$.  Similarly, $\{S_{\ell-1},\varphi_\ell\}\setminus \{\varphi_{\ell-1}\}$ is linearly independent and has the same span as $S_{\ell-1}$.  Thus we can partition $\Phi\setminus \{\varphi_j\}$ into $\ell$ linearly independent sets, say $G=\{G_i\}_{i=1}^{\ell}$ given by
\[
	G_i= \begin{cases}	\{F_j,\varphi_{\ell-1}\}\setminus \{\varphi_{j}\}	& \mbox{for } i=j \\
											\{F_{\ell-1},\varphi_{\ell}\}\setminus \{\varphi_{\ell-1}\}  & \mbox{for } i=\ell-1 \\
											F_{\ell}\setminus \{\varphi_{\ell}\}	& \mbox{for } i=\ell\\
											F_i & \mbox{for } i\neq j, \ell-1, \ell
			\end{cases}
\]
Notice $\abs{G_i}=\abs{F_i}$ and $\spann{}(G_i)=\spann{}(F_i)$ for $i=1,\ldots,\ell-1$.  Then $\{G_{\ell-1},\varphi_j\}$ is linearly independent with $\abs{\{G_{\ell-1},\varphi_j\}}>\abs{F_{\ell-1}}$ contradicting the fact that $F$ was a fundamental partition.
\end{proof}

\begin{corollary} \label{cor-down}
Let $\{F_i\}_{i=1}^{\ell}$ be a fundamental partition of $\Phi=\{\varphi_i\}_{i=1}^M$.  Pick any $\varphi_{\ell} \in F_{\ell}$ and let $S^{(1)}_{{\ell}-1} \subseteq F_{{\ell}-1}$ be smallest set such that $\varphi_{\ell} \in \spann{}(S_{{\ell}-1})$.  Let $S_i^{(1)}\subseteq F_i$, $i=1,\ldots {\ell}-1$ be the smallest subset such that $\spann{}(S^{(1)}_{{\ell}-1}) \subseteq \spann{}(S_i^{(1)})$.  Pick a $S_{j_1}^{(1)}$ such that $\abs{S_{j_1}^{(1)}}\geq \abs{S_{i}^{(1)}}$ for all $i=1,\ldots,{\ell}-1$, and set $S^{(1)}_{j_1}=S^{(2)}_{j_1}$.  Now define $S_i^{(2)}\subseteq F_i$, $i=1,\ldots {\ell}-1$ as the smallest subset such that $\spann{}(S^{(2)}_{j_1}) \subseteq \spann{}(S_i^{(2)})$ and choose $S_{j_2}^{(2)}$ so that $\abs{S_{j_2}^{(2)}}\geq \abs{S_{i}^{(2)}}$ for all $i=1,2\ldots,{\ell}-1$.  Continue this process so that $S_i^{(n)}\subseteq F_i$, $i=1,\ldots {\ell}-1$ is the smallest subset such that $\spann{}(S^{(n)}_{j_{n-1}}) \subseteq \spann{}(S_i^{(n)})$.  Then $\spann{}(S^{(n)}_{j_n})\subseteq \spann{}(F_{i})$, for $i=1,\ldots,{\ell}-1$.
\end{corollary}

\begin{proof}
For $n=1$, this is Lemma \ref{down}.  Notice this gaurantees the sets $S^{(2)}_i$, $i=1,\ldots,{\ell}-1$ are well defined.  It suffices to show $\spann{}(S^{(n)}_{j_n})\subseteq \spann{}(F_{{\ell}-1})$.  Then $S^{(n+1)}_{j_n}=S^{(n)}_{j_n}$, and $S^{(n+1)}_i$, $i=1,\ldots,{\ell}-1$ are well defined.

Suppose instead there existed some $\varphi^{(n)}_{j_n} \in S^{(n)}_{j_n}$ such that $\varphi^{(n)}_{j_n} \notin \spann{}(F_{{\ell}-1})$.  By Proposition \ref{switch}, there exists some $\varphi^{(m_1)}_{j_{n-1}}$, $m_1<n$, such that $\{S^{(n)}_{j_n},\varphi^{(m_1)}_{j_{n-1}} \}\setminus \{\varphi^{(n)}_{j_n}\}$ is linearly independent and has the same span as $S^{(n)}_{j_n}$.  There may be several such vectors $\varphi^{(m_1)}_{j_{n-1}}$, but we may choose a vector such that $m_1$ is minimal.  Indeed simply note if $a<b$ and $j_a=j_b$ then $S^{(a)}_{j_a}\subseteq S^{(b)}_{j_b}$.

Then we consider $S^{(m_1)}_{j_{m_1}}$ and again apply Proposition \ref{switch}.  There exists some $\varphi^{(m_2)}_{j_{m_1-1}}$, $m_2<m_1$, such that $\{S^{(m_1)}_{j_{m_1}},\varphi^{(m_2)}_{j_{m_1-1}} \}\setminus \{\varphi^{(m_1)}_{j_{n-1}}\}$ is linearly independent and has the same span as $S^{(m_1)}_{j_{m_1}}$.  Choose the smallest such $m_2$ for $\varphi^{(m_2)}_{j_{m-1}}$.

By continuing this process $\{m_i\}_{i=1}^{k}$ is a decreasing sequence which terminates with $m_{k}=1$.  One final application of Proposition \ref{switch} implies $\{S^{(1)}_{{\ell}-1},\varphi_{{\ell}}\}\setminus \{\varphi^{(1)}_{{\ell}-1}\}$ is linearly independent and has the same span as ${S^{(1)}_{{\ell}-1}}$.  

Thus we can partition $\Phi \setminus \{\varphi^{(n)}_{j_n}\}$ into ${\ell}$ sets of linear independent vectors, say $G=\{G_i\}_{i=1}^{\ell}$ where $\abs{G_i}=\abs{F_i}$ and $\spann{}(G_i)=\spann{}(F_i)$ for $i=1,\ldots,{\ell}-1$.  However, recalling $\varphi^{(n)}_{j_n} \notin \spann{}(F_{{\ell}-1})$, $\{G_{{\ell}-1},\varphi^{(n)}_{j_n}\}$ is also linearly independent contradicting that $F$ was a fundamental partition.

\end{proof}
 
The arguement in Corollary \ref{cor-down} is quite easy to visualize as shown in figure \ref{fig2}.  Here we take an example where a fundamental partition contains six sets.  The rows correspond to these sets $F_i$, and the $S_i^{(n)}$ are represented by the labeled vectors $\varphi_i^{(n)}$.  When we have an appropriate vector that allows us to apply Proposition \ref{switch} (represented by a shaded square), we can move these vectors as indicated while maintaining linear independence of the row.


\begin{figure}[ht!] 
   \center{\includegraphics*[scale = 0.7]{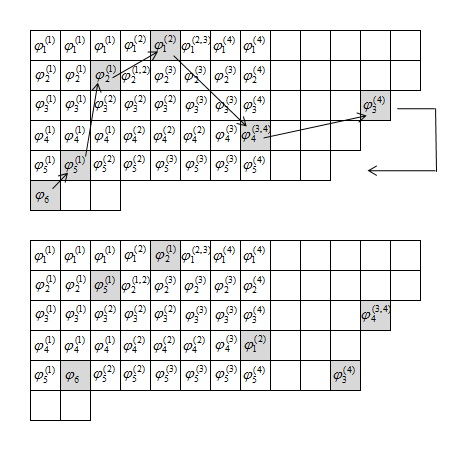}}
   \caption{Original partition and resulting partition after performing argument in corollary \ref{cor-down}}
   \label{fig2}
\end{figure}

Using the above lemmas and corollaries, the Rado-Horn theorem will follow from the existence of so-called transversals in a fundamental partition.  We borrow the term transversal from results in matroid theory where a fundamental partition is a basis for a sum of matroids \cite{transversal,colorings}.

\begin{definition}
Given a fundamental partition $\{F_i\}_{i=1}^{\ell}$ of $\Phi=\{\varphi_i\}_{i=1}^M$ and $t\leq {\ell}$, we call $T\subseteq \Phi$ a {\bf t-transversal} if $T=\{S_i\}_{i=1}^t$, $S_i\subseteq F_i$, and $\spann{}(S_i)=\spann{}(S_j)$ for all $i,j\in \{1,\ldots,t\}$.
\end{definition}

We first show the existence of transversals in a fundamental partition.

\begin{corollary}\label{transversal}
Consider the vectors $\Phi=\{\varphi_i\}_{i=1}^M$ with a fundamental partition $\{F_i\}_{i=1}^{\ell}$.  Fix $t<{\ell}$ and choose any $\varphi_{k}\in F_{k}$ where $t<k$.  Then $\{F_i\}_{i=1}^{t}$ contains a $t$-transversal, $T=\{S_i\}_{i=1}^{t}$, with $\varphi_{k} \in \spann{}(S_i)$ for all $i\in\{1,\dots,t\}$.
\end{corollary}

\begin{proof}
Notice if $F=\{F_i\}_{i=1}^{\ell}$ is a fundamental partition and we remove sets $F_i$, $i=t+1,\ldots, k-1,k+1,\ldots,{\ell}$, then $(\{F_i\}_{i=1}^t,F_{k})$ remains a fundamental partition for its vectors.  It therefore suffices to prove the statement for $t={\ell}-1$, and $k ={\ell}$.

Consider the sets $S_i^{(n)}$, $i=1,\ldots,{\ell}-1$, $n=1,2,\ldots$ as given in Corollary \ref{cor-down} where again $S_{j_n}^{(n)}$ is a largest such set for each $n$.  Notice $\spann{}(S_{j_n}^{(n)})\subseteq \spann{}(S_{j_{n+1}}^{(n+1)})$ for all $i=1,\ldots,{\ell}-1$.  Since we have only finitely many vectors, there exits a $n_0$ such that
\[
	\abs{S_{j_{n_0-1}}^{(n_0-1)}} = \abs{S_{j_{n_0}}^{(n_0)}}.
\]
Then
\[
	\abs{S_{j_{n_0-1}}^{(n_0)}} = \abs{S_{i}^{(n_0)}}. 
\]
Since $\spann{}(S_{j_{n_0-1}}^{(n_0)})\subseteq \spann{}(S_{i}^{(n_0)})$ for all $i=1,\ldots,{\ell}-1$, we conclude
\[  
	\spann{}(S_{j_{n_0-1}}^{(n_0)})=\spann{}(S_{i}^{(n_0)})
\]
for all $i\in 1,\ldots,{\ell}-1$.  Clearly $\varphi_{\ell}\in S_{i}^{(n_0)}$ and $S_i^{(n_0)} \subseteq F_i$ for all $i$ by construction.  Set $S_i=S_{i}^{(n_0)}$, and we have the desired ${\ell}-1$ transversal.
\end{proof}

This argument shows there may be multiple such transversals in a fundamental partition.  It is simple to see that given multiple $t$-transversals, there exists a $t$-transversal containing them.

\begin{lemma}\label{trans lemma}
Let $F=\{F_i\}_{i=1}^{\ell}$ be a fundamental partition  of $\Phi=\{\varphi_i\}_{i=1}^M$.  Suppose $T_1=\{U_i\}_{i=1}^t$ and $T_2=\{V_i\}_{i=1}^t$ are $t$-transversals in $F$.  Then $T=\{U_i \cup V_i\}_{i=1}^t$ is a $t$-transversal.
\end{lemma}

\begin{proof}
For any $i=1,\ldots,t$, $U_i\cup V_i \subseteq F_i$ is linearly independent.  We then have
\begin{align*}
	\spann{}(U_i\cup V_i)&= \spann{}(U_i) + \spann{}(V_i) \\
												&= \spann{}(U_j) + \spann{}(V_j) \\
												&= \spann{}(U_j\cup V_j)
\end{align*}
for all $i,j\in \{1,\ldots,t\}$, and $T$ is a $t$-transversal.
\end{proof}

We are now ready to prove the Rado-Horn theorem.

\section{Proofs of Rado-Horn and its Generalizations}

We begin with the original.

\begin{theorem} (Rado-Horn)
Consider the vectors $\Phi=\{\varphi_i\}_{i=1}^M$.  Then the follwing are equivalent.
\begin{enumerate}[(i)]
	\item \label{one}The set $\Phi$ can be partitioned into sets $\{A_i\}_{i=1}^k$ such that $A_i$ is a linearly independent set for all $i=1,2,\ldots,k$.
	\item \label{two}For any subset $J \subseteq \Phi$, we have $\abs{J}/\dim \spann (J) \leq k$.
\end{enumerate}
\end{theorem}

\begin{proof}
$(\ref{one}\Rightarrow\ref{two}$).  This direction is essentially trivial.  Suppose $\{A_i\}_{i=1}^k$ is a partition of $\Phi$ such that $A_i$ is a linearly independent set for all $i=1,2,\ldots,k$.  Then for any $J\subset \Phi$, let $J_i=J\cap A_i$, $i=1,\ldots,k$.  Then
\[
	\abs{J} = \sum_{i=1}^k \abs{J_i} = \sum_{i=1}^k\dim\spann{}(J_i) \leq k \dim\spann{}(J)
\]
giving the result.

$(\ref{two}\Rightarrow\ref{one})$.  Suppose $\Phi$ cannot be partitioned into $k$ linearly independent sets.  Then for a fundamental partition $\{F_i\}_{i=1}^{\ell}$, we must have $\ell>k$.  By Corollary \ref{transversal}, for any $\varphi \in F_{\ell}$, $\{F_i\}_{i=1}^k$ contains a $k$-transversal, $T$ with $\varphi \in T$.  Then we have
\begin{equation}\label{fail}
	\frac{\abs{T,\{\varphi\}}}{\dim\spann{}(T,\{\varphi\})} = k+\frac{1}{\dim\spann{}(T)}>k.
\end{equation}
\end{proof}

It is a simple matter to adapt the idea of this proof to show a generalized version of the Rado-Horn theorem.  This theorem originally appeared in \cite{CKS}.

\begin{theorem} (Generalized Rado-Horn)
Consider the vectors $\Phi=\{\varphi_i\}_{i=1}^M$.  Then the follwing are equivalent.
\begin{enumerate}[(i)]
	\item \label{one1} There exists a subset $H\subseteq \Phi$ such that $\Phi \setminus H$ can be partitioned into $k$ linearly independent sets.
	\item \label{two1} For any subset $J \subseteq \Phi$, we have $(\abs{J}-\abs{H})/\dim \spann (J) \leq k$.
\end{enumerate}
\end{theorem}

\begin{proof}
$(\ref{one1}\Rightarrow\ref{two1})$.  Suppose such a partition of $\Phi \setminus H$ exists.  Then this direction remains trivial since by for any $J\subseteq \Phi$, we have
\[
	\frac{\abs{J}-\abs{H}}{\dim\spann{}(J)}\leq \frac{\abs{J\setminus H}}{\dim\spann{}(J\setminus H)} \leq k
\]
by the original Rado-Horn theorem.

$(\ref{two1}\Rightarrow\ref{one1})$. For the reverse direction, fix some $L$ and suppose for every $H\subseteq \Phi$ with $\abs{H}=L$, $\Phi \setminus H$ cannot be partitioned into $k$ linearly independent sets.  Take a fundamental partition $F=\{F_i\}_{i=1}^{\ell}$ of $\Phi$.  Then $\ell > k$ and  $\sum_{i=k+1}^{\ell}\abs{F_i}>L$ for otherwise such a partition would exist.  Now consider $G=\{F_i\}_{i=1}^{k+1}$, and notice this is still a fundamental partition for a subset of $\Phi$.  Then Corollary \ref{transversal} and Lemma \ref{trans lemma} imply $G$ contains a $k$-transversal $T=\{T_i\}_{i=1}^{k}$ such that $\spann{}(F_{k+1})\subseteq \spann{}(T)$.

Now consider $J=\{T,F_{k+1},\ldots,F_{\ell}\}$ and $H$ as any subset of cadinality $L$.  Then
\[
	\frac{\abs{J}-\abs{H}}{\dim \spann (J)} > \frac{\abs{T}+1}{\dim \spann (J)} = k+ \frac{1}{\dim \spann (J)}>k.
\]

\end{proof}

We next consider the redundant case.  The following redundant versions of Rado-Horn were also originally proved in \cite{CKS}.  The transversals in a fundamental partition simply explain why the Rado-Horn inequality can fail when $\Phi$ cannot be partitioned into $k$ linearly independent sets.  

\begin{theorem} \label{gen} (Redundant Rado-Horn 1)
Consider the vectors $\Phi=\{\varphi_i\}_{i=1}^M$ in a vector space $V$.  If this set cannot be partitioned into $k$ linearly independent sets, then there exists a partition $\{A_i\}_{i=1}^k$ of the $\Phi$ and a subspace $S$ of $V$ such that the following hold:
\begin{enumerate}[(i)]
	\item For all $1\leq i\leq k$, there exists a subset $S_i\subseteq A_i$ such that $S=\spann{}(S_i)$.
	\item For $J=\{\varphi: \varphi \in S\}$,  $\abs{J}/\dim \spann (J) > k$.
\end{enumerate}
\end{theorem}

\begin{proof}
Take a fundamental partition $F= \{F_i\}_{i=1}^{\ell}$.  Then for $k<\ell$ the hypothesis of Rado-Horn are not met.  Choose any $\varphi \in F_{\ell}$, so there exists a $k$-transveral, $T$, in $F$ which contains $\varphi$ in its span. Simply consider the partition $\{A_i\}_{i=1}^k=(F_1,\ldots,F_{k-1},\{F_{k}, F_{k+1}, \ldots, F_{\ell\}}$ and define the subspace $S=\spann{}(T)$.  Then (1) holds since $T$ is a $k$-transversal, and (2) is clearly true since
\[
	\frac{\abs{J}}{\dim \spann (J)}\geq \frac{\abs{T, \{\varphi\}}}{\dim\spann{}(T, \{\varphi\})}>k
\]
as in equation (\ref{fail}).
\end{proof}

There is one difference between this result when compared to the original.  It is clear from the above arguement that the subspace $S$ may not be unique.  Indeed picking a different $\varphi$ may lead to a different transversal and thus a different subspace.  By taking several transversals and considering their union, we can obtain another result on the subspace $S$ which is more akin to the original theorem \cite{CKS}.

\begin{corollary} (Redundant Rado-Horn 2)
Consider the vectors $\Phi=\{\varphi_i\}_{i=1}^M$ in a vector space $V$.  If this set cannot be partitioned into $k$ linearly independent sets, then there exists a partition $\{A_i\}_{i=1}^k$ of $\Phi$ and a subspace $S$ of $V$ such that the following hold:
\begin{enumerate}[(i)]
	\item \label{one3} For all $1\leq i\leq k$, there exists a subset $S_i\subseteq A_i$ such that $S=\spann{}(S_i)$.
	\item \label{two3} For $J=\{\varphi: \varphi \in S\}$,  $\abs{J}/\dim \spann (J) > k$.
	\item \label{three3} For all $1\leq i\leq k$, $A_i \setminus S_i$ is linearly independent.
\end{enumerate}
\end{corollary}

\begin{proof}
As before, take a fundamental partition $F=\{F_i\}_{i=1}^{\ell}$ of $\Phi$, and consider the partition $\{A_i\}_{i=1}^k=(F_1,\ldots,F_{k-1},\{F_{k}, F_{k+1}, \ldots, F_{\ell}\})$.  We will show there exists a subspace $S$ which satisfies \ref{one3}, \ref{two3}, and \ref{three3} for this partition.  

By Corollary \ref{transversal}, for each $\varphi_i\in F_j$, $j=k+1,\ldots,\ell$, there exits a $k$-transversal, say $T_i$, of $F$ containing $\varphi_i$ in $\spann{}(T_i)$.  By Lemma \ref{trans lemma}, the set
\[
	T=\cup_{\{i:\varphi_i\in F_j, j=k+1,\ldots,l\}}T_i
\]
is a $k$-transversal of $F$ which satisfies $\varphi_i \in \spann{}(T)$ for all $\varphi_i\in F_j$.   Thus
\[
	 \spann{}(F_j) \subseteq \spann{}(T)
\]
for all $j=k+1,\ldots,\ell$.  

Finally, set $S=\spann{}(T)$ with $S_i=T\cap F_i$ for $i=1,\ldots,k-1$ and $S_k=T\cap \{F_{k}, F_{k+1}, \ldots, F_{\ell}\}$.  Then \ref{one3} and \ref{two3} follow as in Theorem \ref{gen} since $T$ is a $k$-transversal which contains in its span at least one $\varphi \in F_j$, $j>k$ (in this case all of them).  Clearly for $i=1\ldots,k-1$, $A_i\setminus S_i \subseteq F_i$ is linearly independent.  Lastly by the way we constructed our transversal,
\begin{align*}
 A_k\setminus S_k &\subseteq \{F_{k}, F_{k+1}, \ldots, F_l\}\setminus \{F_{k+1},\ldots,F_{\ell}\} \\
 &\subseteq F_k
\end{align*}
which is also linearly independent.
\end{proof}

\section{Constructing a Fundamental Partition}

We have shown how the existence of transversals within a fundamental partition leads to an elementary proof of the Rado-Horn theorem.  We will now exhibit how the Rado-Horn inequality may be used to find transversals and construct a fundamental partition.  We begin by describing subsets which maximizes the Rado-Horn inequality.

\begin{lemma} \label{max}
Given vectors $\Phi=\{\varphi_i\}_{i=1}^M$ and a fundamental partition $F=\{F_i\}_{i=1}^{\ell}$, suppose $J\subset \Phi$ maximizes ${\abs{J}}/{\dim \spann{}(J)}$.  Then $J$ is comprised of an $(\ell-1)$-transversal in $F$ together with the set $\{\varphi:\varphi\in F_{\ell},\varphi\in \spann{}(T)\}$.
\end{lemma}

\begin{proof}
Suppose $J$ maximized the Rado-Horn inequality but did not include such an $\ell-1$ transversal.  Now $J$ cannot be partitioned into fewer than $\ell$ linearly independent sets, so consider a fundamental partition $F'=\{F'_i\}_{i=1}^{\ell}$ of $J$.  Then $F'$ contains a maximal $\ell-1$ transversal $T$ where $\spann{}(F'_{\ell})\subseteq \spann{}(T)$.  Define $T'=T\cup F'_{\ell}$. Note then
\[
	\frac{\abs{T'}}{\dim\spann{}(T')}>\ell-1>\frac{\abs{J}-\abs{T'}}{\dim\spann{}(J)-\dim\spann{}(T')}.
\]   
Then
\begin{align*}
	\abs{T'}\dim\spann{}(J)&=\abs{T'}[\dim\spann{}(T')+(\dim\spann{}(J)-\dim\spann{}(T'))]\\
												&>\abs{T'}\dim\spann{}(T')+(\abs{J}-\abs{T'})\dim\spann{}(T')\\
												&=\abs{J}\dim\spann{}(T')
\end{align*}
so $\abs{T'}/\dim\spann(T')>\abs{J}/\dim\spann{}(J)$, a contradiction.
\end{proof}

The previous lemma provides a means for finding a transversal in some unknown fundamental partition.  Next we show that given the proper transversal, if we project our vectors onto the orthogonal complement of the span of the transversal, the nonzero vectors retain their spanning and linearly indpendence properties with respect to each other.  That is, after removing all vectors in the span of the transversal and projecting, the structure of a fundamental partition for the remaining vectors, albeit unknown, is unchanged.  We formalize this claim.

\begin{lemma}\label{construct}
Consider the vectors $\Phi=\{\varphi_i\}_{i=1}^M$ and fundamental partition $F=\{F_j\}_{j=1}^{\ell}$.  For sake of notation, consider $F$ as a partition of the subscripts $\{1,\ldots,M\}$.  Let $t<\ell$ and suppose $T$ is a $t$-transversal of $F$ which satisfies $\spann{}(\{\varphi_i\}_{i\in F_{t+1}})\subseteq \spann{}(\{\varphi_i\}_{i\in T})$.  Let $P_T$ be the orthogonal projection onto $\spann{}(\{\varphi\}_{i\in T})$ and suppose $F'_j=\{i:i\in F_j,i\notin T\}$.  Then $\{F'_j\}_{j=1}^{t}$ is a fundamental partition of $\{(I-P_T)\varphi_i\}_{i\notin T}$.
\end{lemma}

\begin{proof}
First note $\{(I-P_T)\varphi_i\}_{i\in F'_j}$ is linearly independent for each $j=1,\ldots,t$.  Indeed suppose there exists scalars $\{\alpha_i\}_{i\in F'_j}$ such that $\sum_{i\in F'_j} \alpha_i(I-P_T)\varphi_i=0$.  Then $\sum_{i\in F'_j}\alpha_i\varphi_i \in \spann{}(T)=\spann{}(\{\varphi_i\}_{i\in F_j\setminus F'_j})$, but $\{\varphi_i\}_{i\in F_j}$ is linearly independent.  Thus $\alpha_i=0$ for all $i\in F'_j$.

Now suppose these independent sets do not form a fundamental partition.  Then there exists some other partition of $\{i:i\notin T\}$, say $\{A_j\}_{j=1}^{s}$ such that $\{(I-P_T)\varphi_i\}_{i\in A_j}$ is linearly independent for all $j=1,\ldots,s$ and there is some $k$ such that $\abs{A_k}>\abs{F'_k}$ but $\abs{A_i}=\abs{F'_i}$ for all $i<k$.  Setting $F'_j=\emptyset$ for any $t<j\leq s$, it now suffices to show $\{\varphi_i\}_{i\in F_j\setminus F'_j \cup A_j }$ is linearly independent for $j= 1,\ldots,s$, for this would contradict that $F$ was a fundamental partition.

For scalars $\alpha_i$, consider $\sum_{i \in F_j\setminus F'_j \cup A_j }\alpha_i\varphi_i=0$. Under the projection $I-P_T$, this becomes
\[
	\sum_{i \in F_j\setminus F'_j \cup A_j }\alpha_i(I-P_T)\varphi_i= \sum_{i\in A_j}\alpha_i(I-P_T)\varphi_i=0,
\]
and $\alpha_i=0$ for $i\in A_j$.  But then
\[
	\sum_{i \in F_j\setminus F'_j \cup A_j }\alpha_i\varphi_i=\sum_{i\in F_j\setminus F'_j}\alpha_i\varphi_i=0,
\]
and $\alpha_i=0$ for all $i \in F_j\setminus F'_j \cup A_j $.
\end{proof}

We now use the lemmas in this section to construct a fundamental partition.

\begin{center}
\noindent{\bf Construction of a Fundamental Partition}
\end{center}

Let $\Phi=\{\varphi_i\}_{i=1}^M=\{\varphi_{1i}\}_{i=1}^M$ be a collection of vectors (we've added the extra index in order to track an iterative process of projections). Chose $T_1$ as a subset of the indices $\{1,\ldots,M\}$ such that of all subsets of $\Phi$, $J=\{\varphi_i\}_{i\in T_1}$ maximizes
\begin{equation}\label{rh}
	\frac{\abs{J}}{\dim \spann{}(J)}.
\end{equation}
Then by Lemma \ref{max}, $\{\varphi_i\}_{i\in T_1}$ comprises a transversal in some fundamental partition.  Let
\begin{align*}
	t_1&=\dim\spann{}(\{\varphi_i\}_{i\in T_1})\\
	k_1&=\left\lceil \abs{T_1}/t_1\right\rceil\\
	s_1&=\abs{T_1}-(k_1-1)t_1. 
\end{align*}
Then we know exactly how this transversal appears in a fundamental partition.  It is not difficult to see that we may partition $T_1$ as $\{T_{1j}\}_{j=1}^{k_1}$ where 
\begin{enumerate}[(i)]
	\item $\abs{T_{1j}}=t_1$, $j=1,\ldots,k_1-1$
	\item $\abs{T_{1j}}=s_1$, $j=k_1$
	\item $\spann{}(\{\varphi_i\}_{i\in T_{1n}})=\spann{}(\{\varphi_j\}_{j\in T_{1m}})$, $n,m\neq k_1$
	\item $\spann{}(\{\varphi_i\}_{i\in T_{1k_1}})\subseteq \spann{}(\{\varphi_i\}_{i\in T_{1j}})$, $j=1,\ldots,k_1-1$
\end{enumerate}

Let $P_{T_1}$ be the orthogonal projection of $\Phi$ onto $\spann{}(\{\varphi_{1i}\}_{i\in T_1})$. Define $\Phi_2=\{(I-P_{T_1})\varphi_{1i}\}_{i\notin T_1}=\{\varphi_{2i}\}_{i\notin T_1}$.  Now chose $T_2$ as a subset of the indices in $\{1,\ldots,M\} \setminus T_1$ such of all subsets of $\Phi_2$, $J=\{\varphi_{2i}\}_{i\in T_2}$ maximizes (\ref{rh}).  Then $\{\varphi_{2i}\}_{i\in T_2}$ comprises a transversal in a fundamental partition of $\Phi_2$.  Let
\begin{align*}
	t_2&=\dim\spann{}(\{\varphi_{2i}\}_{i\in T_2})\\
	k_2&=\left\lceil \abs{T_2}/t_2\right\rceil\\
	s_2&=\abs{T_2}-(k_2-1)t_2,
\end{align*}
and we may partition $T_2$ as $\{T_{2i}\}_{i=1}^{k_2}$ where
\begin{enumerate}[(i)]
	\item $\abs{T_{2i}}=t_2$, $i=1,\ldots,k_2-1$
	\item $\abs{T_{2i}}=s_2$, $i=k_2$
	\item $\spann{}(\{\varphi_j\}_{j\in T_{2n}})=\spann{}(\{\varphi_j\}_{j\in T_{2m}})$, $n,m\neq k_2$
	\item $\spann{}(\{\varphi_j\}_{j\in T_{2k_2}})\subseteq \spann{}(\{\varphi_j\}_{j\in T_{2i}})$, $i=1,\ldots,k_2-1$.
\end{enumerate}

We continue so that $P_{T_j}$ is the orthogonal projection of $\Phi_j$ onto $\spann{}(\{\varphi_{ji}\}_{i\in T_{j}})$. Define $\Phi_{j+1}=\{(I-P_{T_j})\varphi_{ji}\}_{i\notin T_1\cup\ldots\cup T_{j}}=\{\varphi_{(j+1)i}\}_{i\notin T_1\cup\ldots\cup T_{j}}$.  Now choose $T_{j+1}$ as a subset of the indices $\{1,\ldots,M\} \setminus \{T_1\cup\ldots\cup T_{j}\}$ such that of all subsets of $\Phi_{j+1}$, $J=\{\varphi_{(j+1)i}\}_{i\in T_{j+1}}$ maximizes (\ref{rh}).  Then $\{\varphi_{(j+1)i}\}_{i\in T_{j+1}}$ comprises a transversal in a fundamental partition of $\Phi_{j+1}$.  Letting
\begin{align*}
	t_{j+1}&=\dim\spann{}(\{\varphi_{({j+1})i}\}_{i\in T_{j+1}})\\
	k_{j+1}&=\left\lceil \abs{T_{j+1}}/t_{j+1}\right\rceil\\
	s_{j+1}&=\abs{T_{j+1}}-(k_{j+1}-1)t_{j+1}. 
\end{align*}
may partition $T_{j+1}$ as $\{T_{({j+1})i}\}_{i=1}^{k_{j+1}}$
\begin{enumerate}[(i)]
	\item $\abs{T_{({j+1})i}}=t_{j+1}$, $i=1,\ldots,k_{j+1}-1$
	\item $\abs{T_{({j+1})i}}=s_{j+1}$, $i=k_{j+1}$
	\item $\spann{}(\{\varphi_j\}_{j\in T_{({j+1})n}})=\spann{}(\{\varphi_j\}_{j\in T_{({j+1})m}})$, $n,m\neq k_{j+1}$
	\item $\spann{}(\{\varphi_j\}_{j\in T_{({j+1})k_{j+1}}})\subseteq \spann{}(\{\varphi_j\}_{j\in T_{({j+1})i}})$, $i=1,\ldots,k_{j+1}-1$
\end{enumerate}

Notice $k_j\geq k_{j+1}$.  There is a small techinical issue here in that if $k_j=k_{j+1}$, the hypotheses of Lemma \ref{construct} are not met, and we require this Lemma to gaurantee this iterative process leads to a fundamental partition.  If $k_j=k_{j+1}$, we may redefine $T_j$ as the larger transversal $T'_j=T_j \cup T_{j+1}$.  Then recalculate $t_j$ and $s_j$ ($k_j$ will remain unchanged) and continue.  The transversal becomes larger until at some point $k_j>k_{j+1}$.  An exception is when $k_j=0$, but this is not an issue since $k_j=0$ only when $\Phi_j=\emptyset$.

Suppose $r$ is such that $k_r\neq 0$ but $k_{r+1}=0$.  Finally, for $i>k_j$ adopt the convention $T_{ji}=\emptyset$.  

Then $F=\{F_i\}_{i=1}^{k_1}$ where
\[
	F_i = \cup_{j=1,\ldots,r}\{\varphi_{\ell}\}_{\ell \in T_{ji}}
\]
for $i=1,\ldots,k_1$ is a fundamental partition of $\Phi$.  

The fact that piecing together the vectors from the transversals of projections yields a fundamental partition follows imediately from Lemma \ref{construct}, where we showed such projections do not change the structure of a fundamental partition.  
See figure \ref{fig3} for an example of a fundmental partition showing values $t_i, k_i, s_i$.


\begin{figure}[ht!] 
   \center{\includegraphics*[scale = 0.7]{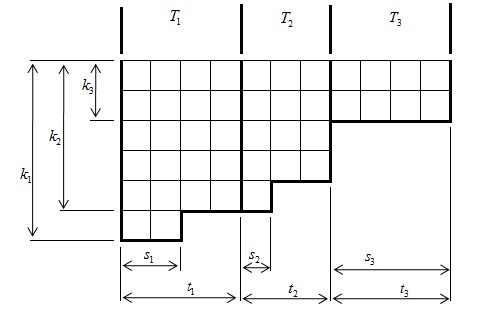}}
   \caption{Fundamental partition constructed from transversals of appropriate projections}
   \label{fig3}
\end{figure}

\begin{remark}
By constructing a fundamental partition we have essentially used the Rado-Horn inequality to described many of the spanning properties of the vectors.  For example, using the notation from the above construction, a collection of vectors $\Phi=\{\varphi_i\}_{i=1}^M$ span a $\sum_{i=1}^rt_i$-dimensional space and can be partitioned into at most $k_r$ spanning sets when $t_r=s_r$ and at most $k_r-1$ spanning sets when $t_r\neq s_r$.
\end{remark}

\end{document}